\documentclass[11pt,reqno]{amsart}

\usepackage{amsmath,mathtools}
\usepackage{amsthm}
\usepackage{amssymb}
\usepackage{verbatim}
\usepackage[dvipsnames]{xcolor}
\usepackage{hyperref}
\usepackage{mathtools}
\usepackage{enumitem} 
\usepackage{stmaryrd}
\usepackage{soul}
\usepackage[numbers,sort&compress]{natbib}
\usepackage{fullpage}
\usepackage{caption}
\def\thin{\hspace{.5pt}}
\usepackage[backgroundcolor=none]{todonotes}

\theoremstyle{definition}

\newtheorem{theorem}{Theorem}
\newtheorem*{theorem*}{Theorem}
\newtheorem*{conj}{Conjecture}

\newtheorem*{remark}{Remark}

\newtheorem{lemma}[theorem]{Lemma}

\newtheorem{corollary}[theorem]{Corollary}
\newtheorem*{thm}{Theorem}

\newtheorem*{claim}{Claim}

\newcommand{\R}{\mathbb{R}}
\newcommand{\Q}{\mathbb{Q}}

\newcommand{\Z}{\mathbb{Z}}
\newcommand{\M}{\mathcal{M}}
\newcommand{\MM}{\mathrm{M}}
\newcommand{\EE}{\mathrm{E}}
\newcommand{\Li}{\operatorname{Li}}

\tikzset{circ/.style = {fill, circle, inner sep = 0, minimum size = 3}}

\newcommand{\re}{\textnormal{Re}}
\newcommand{\im}{\textnormal{Im}}

\renewcommand{\=}{\;=\;}

\renewcommand{\Re}{\re}
\renewcommand{\Im}{\im}

\usepackage{color}

\usepackage{commath}
\usepackage{float}
\usepackage{enumitem}

\DeclareMathOperator{\dist}{dist}
\DeclareMathOperator{\sgn}{sgn}

\tikzset{circ/.style = {fill, circle, inner sep = 0, minimum size = 3}}
\usetikzlibrary{arrows.meta}
\usetikzlibrary{decorations.markings}
\usetikzlibrary{decorations.pathmorphing}
\usetikzlibrary{positioning}
\usetikzlibrary{intersections}
\usepackage{tkz-euclide}
\pgfarrowsdeclarecombine{twolatex'}{twolatex'}{latex'}{latex'}{latex'}{latex'}
\pgfarrowsdeclarecombine{twolatex'}{twolatex'}{latex'}{latex'}{latex'}{latex'}
\tikzset{->-/.style = {decoration={markings,
 mark=at position #1 with {\arrow[scale=2]{latex'}}},
 postaction={decorate}}}
\tikzset{-<-/.style = {decoration={markings,
 mark=at position #1 with {\arrowreversed[scale=2]{latex'}}},
 postaction={decorate}}}

\newcommand{\bea}{\begin{equation}\begin{aligned}}
\newcommand{\eea}{\end{aligned}\end{equation}}
\allowdisplaybreaks

\title{Exceptional Sign Pairs in Oscillatory Asymptotics and Conjectures of Andrews}
\author[Kalita]{Jayashree Kalita}
\address{Department of Mathematics,
	Vanderbilt University, Nashville, Tennessee, United States}
\email{jayashree.kalita@vanderbilt.edu}

\begin{document}

\begin{abstract}
In this paper, we prove an \emph{exceptional sign pair phenomenon} for a general family of integer sequences with exponentially growing oscillatory asymptotics.
This resolves, in particular, conjectural sign patterns proposed by Andrews in a 1986 paper for three $q$-series from Ramanujan's Lost Notebook, namely $v_2(q), v_3(q),$ and $v_4(q)$.
Recent work of Kundu, Storzer, Wang, and the author established that the coefficients of these $q$-series are alternating in sign except in a density-zero set. 
We prove here that the exceptional indices where the alternating sign pattern fails occur infinitely often in a structured manner. 
More precisely, we show that there exist infinitely many pairs of consecutive coefficients having the same sign, and that the magnitude of at least one coefficient in each pair is a local minimum of the sequence of absolute values of the coefficients. 
Together with earlier results, this completes the resolution of Andrews' conjectures and their analogs for all the $q$-series in his original study.
\end{abstract}

\maketitle
\section{Introduction}

    In~\cite{andrews1986questions}, Andrews studied five $q$-hypergeometric series from Ramanujan's Lost Notebook~\cite{RlostIV,RlostV} and observed striking growth and sign behavior in their coefficients. 
    These five series are
\begin{align*}
    \sigma(q)
    &
    \;\coloneqq\; \sum_{n\geq0}\frac{q^{n(n+1)/2}}{(-q;q)_n} 
    = 1+q-q^2+2q^3-2q^4+q^5+\cdots
    \,=\,\sum_{n\geq 0} S(n)\, q^n,\\
    v_1(q)
    &
    \;\coloneqq\;
    \sum_{n\geq0}\frac{q^{{n(n+1)}/2}}{(-q^2;q^2)_n} 
    \= 1+q+q^6-q^7-q^8+q^9+\cdots
    \;=\;\sum_{n\geq 0} V_1(n)\thin q^n,\\
    v_2(q) 
    &
    \;\coloneqq\;
    \sum_{n\geq1}\frac{q^{2n^2-n}}{(-q;q^2)_n} 
    \= q-q^2+q^3-q^4+q^5-q^9+\cdots
    \;=\;\sum_{n\geq 0} V_2(n)\thin q^n,\\
    v_3(q)
    &
    \;\coloneqq\;
    \sum_{n\geq0}\frac{(-1)^nq^{n(n+1)/2}}{(-q;q)^2_n}
    \= 1-q+2q^2-2q^3+2q^4+\cdots
    \;=\;\sum_{n\geq 0} V_3(n)\thin q^n,\\
    v_4(q)
    &
    \;\coloneqq\;
    \sum_{n\geq0}\frac{(-1)^nq^{2n^2}}{(-q;q^2)_{2n}}
    \= 1-q^2+q^3-q^4+2q^5-2q^6+\cdots
    \;=\;
    \sum_{n\geq 0} V_4(n)\thin q^n.
\end{align*}

    Based on numerical evidence, Andrews made six conjectures on the coefficients of these $q$-series. 
    His first two conjectures addressed the series $\sigma(q)$ and stated that
    \begin{equation*}
    \limsup_{n\to\infty}\abs{S(n)} = \infty \text{ and } S(n)=0 \text{ for infinitely many }n.
    \end{equation*}
    Andrews--Dyson--Hickerson~\cite{andrews1988partitions} famously proved these two conjectures, uncovering a remarkable connection between the coefficients of $\sigma(q)$ and the arithmetic of $\Q(\sqrt{6})$. 
    This example has since influenced a large body of work, including Cohen's~\cite{cohen1988q} construction of a related Maaß wave form, Zwegers'~\cite{zwegersmock} theory of mock Maaß theta functions, Zagier's~\cite{zagier2010quantum} theory of quantum modular forms, and more recently, the work of Li--Roehrig \cite{li2025mock}.

    For the $q$-series $v_1(q),v_2(q),v_3(q),\text{ and }v_4(q)$, Andrews observed what he described as ``great sign regularity." 
    More precisely, Andrews made the following conjectures for $v_1(q)$. 
    He stated that Conjecture 3~\cite{andrews1986questions} and appropriate analogs of Conjectures 4--6~\cite{andrews1986questions} hold for $v_2(q),v_3(q),\text{ and }v_4(q)$ as well.
    \begin{conj}(Conjecture 3~\cite{andrews1986questions})
    We have that $\abs{V_1(n)}\to\infty$ as $n\to\infty$.
    \end{conj}
    \begin{conj}(Conjecture 4~\cite{andrews1986questions})
    For almost all $n$, $V_1(n),V_1(n+1),V_1(n+2),\text{ and }V_1(n+3)$ are two positive and two negative numbers.    
    \end{conj}
    \begin{conj}(Conjecture 5~\cite{andrews1986questions})
    \sloppy For $k\ge5$, there is an infinite sequence $N_5=293,N_6=410,$ $N_7=545,N_8=702,\ldots,N_k\ge10k^2,\ldots$ such that $V_1(N_k),V_1(N_k+1),V_1(N_k+2)$ all have the same sign.
    \end{conj}
    \begin{conj}(Conjecture 6~\cite{andrews1986questions}) \sloppy With reference to Conjecture 5~\cite{andrews1986questions}, the numbers $\abs{V_1(N_k)},$ $ \abs{V_1(N_k+1)},\abs{V_1(N_k+2)}$ contain a local minimum of the sequence $(\,\abs{V_1(n)})_{n\ge0}$.   
    \end{conj}
    
    Conjectures~3 (specifically, a refined density-one version) and~4 were proved for $v_1(q)$ by Folsom--Males--Rolen--Storzer~\cite{folsom2023oscillating}, and their analogs for $v_2(q),v_3(q),\text{ and }v_4(q)$ were proved by Kundu, Storzer, Wang, and the author~\cite{kalita2026}. 
    Both works relied on establishing precise asymptotic formulas for the underlying coefficients. More precisely, the later work established the following theorem.
    
    \begin{thm}(Theorem A~\cite{kalita2026})
    The sequences $V_2(n),V_3(n)$, and $V_{4}(n)$ satisfy an \emph{almost alternating sign pattern}, meaning that the following assertions are true for $j=2, 3, 4$.
    \begin{enumerate}[label=\textup{(}\roman*\textup{)}]
    \item The sequence $\abs{V_j(n)} \to \infty$ as $n\to \infty$ away from a set of density $0$.
    \item For almost all $n$, the coefficients $V_j(n)$ and $V_j(n+1)$ have opposite signs.
    \end{enumerate}  
    \end{thm}
    The theorem was proved using the following asymptotic formula for the coefficients, which will also play a key role in the present paper.

    Recall the Bloch--Wigner dilogarithm $D(e^{i\theta}) \coloneqq \Im(\Li_2(e^{i\theta}))$, and more generally, $D(x) := \Im(\Li_2(x))+\arg (1-x) \log\abs{x}$ (see, for example, \cite{zagier2007dilogarithm}).
    \begin{thm}(Theorem B~\cite{kalita2026})
    For $j\in \{2,3,4\}$, as $n\to\infty$, the following asymptotics hold
    \begin{align*}
    V_j(n)
    &\= (-1)^{n}\alpha_j\frac{e^{2\sqrt{n}\Re\left(\sqrt{W_j}\right)}}{\sqrt{n}}
    \cos\left(\!2\sqrt{n}\Im\bigl(\sqrt{W_j}\bigr)+\frac{\pi}{4}\!\right)
    \left(1+O\!\left(n^{-\frac{1}{2}}\right)\right)
    + O\!\left(\frac{e^{\sqrt{n}\Re\left(\sqrt{W_j}\right)}}{\sqrt{n}}\right)\!,
    \end{align*}
    where
    \vspace{-0.5cm}
    \begin{align*}
    &\alpha_2 \= \frac 1{\sqrt{2}\sqrt[4]{3}},\qquad\qquad\qquad\quad
    &&W_2 \= -\frac{\pi^2}{24}+D(e^{\pi i/3})\frac i2,\\
    &\alpha_3 \= -\frac{1}{\sqrt[4]{3}},\qquad\qquad
    &&W_3 \= D(e^{\pi i/3})\frac i2,\\
    &\alpha_4 \= \frac1{2\sqrt[4]{3}},\qquad\qquad
    &&W_4 \= D(e^{\pi i/3})\frac i2,
    \end{align*}
    and $D(e^{\pi i/3})=1.014942\cdots$.
    \end{thm}

    While Conjectures 3 and 4 describe the prevailing growth and sign regularity of the coefficients, Conjectures 5 and 6 concern the exceptional behavior that occurs within this regular pattern. 
    These latter conjectures remained substantially less understood until very recently.
    Folsom--Males--Rolen--Storzer~\cite[Section 6]{folsom2023oscillating} gave a conditional explanation of Conjectures~5 and~6, contingent on deep unresolved irrationality questions.
    Bachraoui~\cite{Bach2026} subsequently provided an unconditional proof of these conjectures for $v_1(q)$ by analyzing the asymptotic formula for the coefficients $V_1(n)$ derived in~\cite{folsom2023oscillating}. 
    The aim of the present paper is to prove the analogs of Conjectures~5 and~6 for $v_2(q),v_3(q),\text{ and }v_4(q)$, and more generally to establish an exceptional sign pair phenomenon for a broad class of integer sequences with prescribed oscillatory asymptotics.
    This completes the resolution of all of Andrews' conjectures and their analogs for these five $q$-series.

\begin{table}[ht]
    \centering
    {\fontsize{6.5}{14}\selectfont
    \setlength{\tabcolsep}{3pt}

    \begin{minipage}[t]{0.32\textwidth}
    \centering
    \begin{tabular}{|c|c|c|c|c|}
    \hline
    $n$ & $V_2(n-1)$ & $V_2(n)$ & $V_2(n+1)$ & $V_2(n+2)$\\
    \hline
    48  & 2    & -1   & -1   & 1\\
    83  & 4    & -2   & -1   & 4\\
    127 & -13  & 6    & 3    & -11\\
    \hline
    \end{tabular}
    \end{minipage}
    \hfill
    \begin{minipage}[t]{0.32\textwidth}
    \centering
    \begin{tabular}{|c|c|c|c|c|}
    \hline
    $n$ & $V_3(n-1)$ & $V_3(n)$ & $V_3(n+1)$ & $V_3(n+2)$\\
    \hline
    49  & -15  & 6    & 8   & -20\\
    102 & -174 & 67   & 14   & -107\\
    175 & -2348  & 1067 & 330  & -1771\\
    \hline
    \end{tabular}
    \end{minipage}
    \hfill
    \begin{minipage}[t]{0.32\textwidth}
    \centering
    \begin{tabular}{|c|c|c|c|c|}
    \hline
    $n$ & $V_4(n-1)$ & $V_4(n)$ & $V_4(n+1)$ & $V_4(n+2)$\\
    \hline
    49 & 6 & -2 & -3 & 9\\
    102 & 87    & -42  & -8    & 62\\
    175 & 1160  & -524 & -153  & 877\\
    \hline
    \end{tabular}
    \end{minipage}
    }
    \medskip
    \caption{The first three exceptional sign pairs with their adjacent coefficients.}\label{Table}
\end{table}

     Fix $j\in\{2,3,4\}$. Numerical computations (see Table~\ref{Table}) reveal the existence of consecutive coefficients of $v_j(q)$ that share the same sign. We refer to the pair $(V_j(n),V_j(n+1))$ as an exceptional sign pair for $V_j$ if $V_j(n)$ and $V_j(n+1)$ have the same sign. The occurrence of such exceptional sign pairs necessitates the qualifier \emph{almost} in Theorem~A~\cite{kalita2026}. Our main result for these $q$-series is the following theorem which shows that such indices occur infinitely often and satisfy a local minimum property. 
    
\begin{theorem}\label{thm:main}
    The sequences $V_2(n),V_3(n)$, and $V_4(n)$ exhibit an \emph{exceptional sign pair phenomenon} in the following sense. For each fixed $j\in\{2,3,4\}$, 
    \begin{enumerate}[label=(\roman*)]
        \item There exists an infinite sequence $(N_k)_{k\ge1}$ (depending on $j$) with $N_k \ge 10k^2$ such that $V_j(N_k)\text{ and }V_j(N_k+1)$ have the same sign.
        \item At least one of $\abs{V_j(N_k)}$ and $\abs{V_j(N_k+1)}$ is a local minimum of the sequence $(\,\abs{V_j(n)})_{n\ge0}$.
    \end{enumerate}
\end{theorem}
    We prove Theorem~\ref{thm:main} in a broader setting that places the exceptional sign pair phenomenon within a unified framework.
    More precisely, Theorem~\ref{thm:Gen} in Section~\ref{sec:final} establishes a general result on exceptional sign pairs for the class of sequences satisfying the following asymptotics, from which Theorem~\ref{thm:main} follows immediately. 
    Fix $\lambda\in\{2\Im\bigl(\sqrt{W_j}\bigr): j=2,3,4\}$ and set $\phi=\pi/4$. 
    We retain $\phi$ in the notation since, when $\pi^2/\lambda^2$ is irrational, our results extend to arbitrary $\phi\in\mathbb{R}$.
    Let $V(n)$ be a sequence of integers satisfying
    \begin{equation}\label{V_n}
     V(n) \= (-1)^{n}\alpha\frac{e^{\beta\sqrt{n}}}{\sqrt{n}}
    \cos\bigl(\lambda\sqrt{n}+\phi)
    \left(1+O(n^{-\frac{1}{2}})\right)+O\left(\frac{e^{\gamma\sqrt{n}}}{\sqrt{n}}\right), 
    \end{equation}
    as $n\to\infty$, where $\alpha\neq0,\beta,\gamma$ are real numbers with $\beta>\gamma>0$.

    The sign behavior of $V(n)$ is governed by the asymptotics in~\eqref{V_n}. 
    We refer to the first term as the dominant term.
    Whenever this term dominates the error, the sign of $V(n)$ is determined, up to the fixed factor $\sgn(\alpha)$, by $(-1)^n$ and the cosine term. 
    Near a zero of the cosine, its sign change cancels the sign change coming from the $(-1)^n$ factor, giving rise to the exceptional sign pairs.
    More precisely, if there is a zero of the cosine factor between $\sqrt n$ and $\sqrt{n+1}$ and it remains sufficiently separated from these endpoints, then the dominant term determines the signs of both $V(n)$ and $V(n+1)$, which are consequently the same.
    The uniform separation theorem established in Section~\ref{sec:Sep Thm} guarantees the existence of infinitely many such zeros.
    The local minimum property follows by comparing the magnitudes of the dominant terms at the exceptional indices with those at their respective neighboring indices, and then controlling the error.
    
    The general framework considered here is structurally distinct from the setting of Bachraoui~\cite{Bach2026} for $v_1(q)$.
    In particular, $V_1(n)$ lies outside the general family underlying our results.
    This structural difference is reflected in the resulting sign behavior.
    In Bachraoui's setting, the asymptotic formula separates according to parity and involves two trigonometric factors, leading to infinitely many blocks of three consecutive coefficients having the same sign.
    In our setting, the dominant term contains a single oscillatory term along with the $(-1)^n$ sign factor. At the level of the dominant term, this structure gives rise to exceptional pairs rather than longer blocks, in contrast with the same sign triples arising in Bachraoui's case.
    Indeed, a block of three consecutive coefficients having the same sign would require the cosine to change sign between both $\sqrt n$ and $\sqrt{n+1}$ as well as $\sqrt{n+1}$ and $\sqrt{n+2}$ for some $n$, hence would require two of its zeros to lie in the interval $(\sqrt n,\sqrt{n+2})$. 
    This cannot occur infinitely often, since consecutive zeros of $\cos(\lambda x+\phi)$ are separated by $\pi/\lambda$, which exceeds $\sqrt{n+2}-\sqrt n$ for all sufficiently large $n$. 

    With these conjectures and their analogs resolved, it is natural to ask whether these $q$-series reflect a deeper automorphic structure, a question that we leave for future work.

    The rest of the paper is organized as follows. In Section~\ref{sec:Sep Thm}, we prove a uniform separation theorem which provides the key ingredient for proving the existence of infinitely many exceptional sign pairs. Section~\ref{sec:osc lemmas} is devoted to establishing several lemmas on the behavior of the oscillatory term in the asymptotic expansion of $V(n)$ as given in~\eqref{V_n}.
    Finally, in Section~4, we prove our general exceptional sign pair theorem, Theorem~\ref{thm:Gen}, and then specialize it to the $q$-series $v_2(q),v_3(q),$ and $v_4(q)$ to obtain Theorem~\ref{thm:main}.
    \medskip
    
    \paragraph*{\it{Acknowledgements}}
    The author is grateful to Kathrin Bringmann and Larry Rolen for many helpful discussions and valuable comments. 
    The author also thanks George Andrews for his kind encouragement, and Amanda Folsom, Debanjana Kundu, Jyotirmoy Sengupta, and Wadim Zudilin for helpful suggestions.

\section{The Oscillatory Factor and a Uniform Separation Theorem}\label{sec:Sep Thm}

    Denote the oscillatory factor in the asymptotics of $V(n)$ by
    \begin{equation}\label{F}
    F(x)\coloneqq\cos(\lambda x + \phi),
    \end{equation}
    for $x\ge0$.
    The zeros of $F$ are 
    \begin{alignat}{2}
    s_m
    &\coloneqq \frac{\pi-2\phi+2m\pi}{2\lambda},
    &\qquad & m \in \Z_{\geq 0},
    \label{s_m}\\
    \intertext{and we define}
    t_m
    &\coloneqq s_m^2.
    &&\label{t_m}\\
    \intertext{Observe that $(t_m)_{m\geq 0}$ is a strictly increasing sequence of non-negative real numbers satisfying}
    t_m
    &\to \infty,
    &\qquad \text{as } & m \to \infty.
    \label{t_m:to-infty}
    \end{alignat}
    
    The main result of this section is the following uniform separation     theorem. 
    It shows that the quadratic sequence $(t_m)_{m\ge0}$ associated with the positive zeros of the oscillatory factor remains uniformly separated from the integers along a subsequence.  

\begin{theorem}\label{thm:M and delta} 
    There exist $\delta>0$ and an infinite set $\M\subset \Z_{\ge 0}$ such that
    \begin{equation}\label{lem:dist}
       \dist(t_m,\Z) \geq \delta, \qquad\qquad \forall m \in \M
    \end{equation}  
\end{theorem}
\begin{proof}
    Write 
    \[
    t_m = \kappa(m+\mu)^2,
    \]
    where $\kappa\coloneqq\tfrac{\pi^2}{\lambda^2}$ and $\mu\coloneqq\tfrac 1 2 - \tfrac \phi \pi$. 
    Define the polynomial $P(x)$ by
    \[
    P(x)\coloneqq\kappa(x + \mu)^2.
    \]
    
    \textbf{Case I. $\kappa$ is irrational:} Similar to~\cite{Bach2026}, $P(x)$ is a quadratic polynomial whose leading coefficient is irrational. 
    By Theorem 3.2 in~\cite[p.~27]{KuipN}, the sequence $(P(m))_{m\geq 0}$ is equidistributed modulo $1$.
    Hence, infinitely many $m$ satisfy 
    \[
    P(m)\bmod 1\in\left[\frac{1}{4},\frac{3}{4}\right],
    \]
    and for such $m$,
    \[
    \dist(P(m),\Z)\geq\frac{1}{4}.
    \]
    That is, the inequality~\eqref{lem:dist} holds for $\delta=\tfrac{1}{4}$ and a suitable infinite set $\M\subset\Z_{\ge0}$.

    \medskip
    
    \textbf{Case II. $\kappa$ is rational:} Write $\kappa = \tfrac{a}{b}$ with $a,b\in\Z,\,b>0\,\text{and}\,\gcd(a,b)=1$.
    By assumption, $\phi=\tfrac{\pi}{4}$ and so $\mu = \tfrac{1}{4}$.
    Thus,
    \[
    P(m) = \frac{a}{16b}(4m+1)^2,
    \]
    which gives
    \[
    P(m+4b)-P(m) = 2a(4m+1)+16ab \in \Z.
    \]
    Hence,
    \[
    P(m+4b)\equiv P(m)\bmod{1}, \qquad \forall m\in\Z_{\geq 0}.
    \]
    That is, the sequence $(P(m))_{m\geq 0}$ is periodic modulo $1$ with period $4b$.
    
    \begin{claim}
     There is at least one residue class $r$ modulo $4b$ such that $P(r)\notin\Z$.
    \end{claim}
    To the contrary, assume $P(m)\in\Z,\,\forall m\in\Z_{\geq 0}$. 
    Then $\kappa/16 = P(0)\in\Z$ implies $\kappa\in\Z$.
    To complete the proof of the claim, it suffices to show that $\kappa\notin\Z$. 
    This follows from~\cite[Lemma 1]{Bach2026} which proves that 
    \begin{equation}\label{ineq:BW}
    1.0147430670 < D(e^{\pi i/3}) < 1.0149758071
    \qquad
    \text{and}
    \qquad
    \frac{4\pi^2}{D(e^{\pi i/3})}\notin\Z.
    \end{equation}
    Using~\eqref{ineq:BW}, we obtain the following.
    \begin{itemize}
        \item If $\lambda = 2\Im\sqrt{W_2}$, then using the inequality in~\eqref{ineq:BW}, we find
        \begin{equation}\label{ineq:lambda 1}
           4.6361272<\kappa = \frac{\pi^2}{(2\Im\sqrt{W_2})^2} = \frac{1}{\sqrt{\tfrac{1}{144}+\tfrac{D(e^{\pi i/3})^2}{\pi^4}}+\tfrac{1}{12}}<4.6365211,
        \end{equation}
        which implies $\kappa\notin\Z$.
        \item If $\lambda = 2\Im\sqrt{W_3} = 2\Im\sqrt{W_4} = \sqrt{D(e^{\pi i/3})}$, then $\kappa = \tfrac{\pi^2}{D(e^{\pi i/3})}\notin\Z$; otherwise $\tfrac{4\pi^2}{D(e^{\pi i/3})}=4\kappa\in\Z$ contradicting~\eqref{ineq:BW}. 
    \end{itemize}
    This completes the proof of the claim.

    We may now consider the non-empty finite set 
    \[
    S\coloneqq \{\dist(P(r),\Z): 0\leq r < 4b,\, P(r)\notin\Z\},
    \]
    and let $\delta\coloneqq \min S > 0$.
    Also define
    \[
    \M\coloneqq\{r+4kb: 0\leq r < 4b,\,P(r)\notin\Z,\,k\in\Z_{\geq0}\}.
    \]
    It follows from the claim that $\M$ is infinite. Moreover, if $m = r+4kb\in \M$, then $P(m)\equiv P(r)\bmod 1$,
    and hence
    \[
    \dist(P(m),\Z) = \dist(P(r),\Z)\ge \delta.
    \]
    This completes the proof.
\end{proof}

\begin{remark}
    The first case of the proof remains valid for any $\lambda\in\R^*$ such that $\pi^2/\lambda^2$ is irrational and for any $\phi\in\mathbb{R}$. Consequently, Theorem~\ref{thm:M and delta} holds under these assumptions.
\end{remark}

\section{Estimates Near the Zeros of the Oscillatory Factor}\label{sec:osc lemmas}

    In this section, we analyze the behavior of the oscillatory factor near its zeros.
    We begin by identifying the integers whose square roots lie nearest to these zeros, which are later shown to be responsible for the exceptional sign pairs in our main result. 
    Specifically, the sign pattern of the oscillatory factor is determined at the square roots of these core integers, and we then extend our framework to their immediate neighbors to obtain the necessary size estimates.

    Throughout the remainder of the paper, we fix $\M$ and $\delta$ as in Theorem~\ref{thm:M and delta} and adopt the following notation.
    Write \begin{alignat}{2} 
    & x_n := \lambda\sqrt{n}+\phi,
    &\qquad & n \in \mathbb{Z}_{\geq 0},\label{x_n}
    \intertext{and set}
    & N_m := \lfloor t_m \rfloor,
    &\qquad & m \in \M.
    \end{alignat}
    By the definition of the floor function, $N_m$ is the unique integer satisfying 
    \begin{equation}\label{Def:Floor}
     N_m \le t_m <N_m+1.   
    \end{equation}
    By~\eqref{t_m:to-infty}, as $m\to\infty$ with $m\in\M$, we have $t_m\to\infty$ and hence $N_m\to\infty$. By discarding finitely many elements of $\M$, if necessary, we may assume that $N_m>1$ for all $m\in\M$.
    
\begin{lemma} \label{lem:ineq} 
    For each $m\in \M$, the following inequalities hold:
    \begin{enumerate}[label=(\alph*)]
        \item We have $N_m < t_m < N_m+1$,
        \item We have $x_{N_m} < \lambda s_m+\phi < x_{N_m+1}$,
        \item We have $\delta \le t_m-N_m \le 1-\delta$.
    \end{enumerate}
\end{lemma}

\begin{proof}
    Since $m\in \M$, it follows from Theorem~\ref{thm:M and delta} that $t_m\notin\Z$. 
    Together with~\eqref{Def:Floor}, we obtain the strict inequalities 
    \begin{equation}\label{eqn:Nmtm}
    N_m = \lfloor t_m \rfloor < t_m < \lfloor t_m\rfloor+1 =N_m+1.
    \end{equation}
    Because the function $u\mapsto\lambda\sqrt{u}+\phi$ is strictly increasing on $(0,\infty)$, it follows from~\eqref{eqn:Nmtm} that
    \[
    x_{N_m} < \lambda\sqrt{t_m}+\phi < x_{N_m+1}.
    \]
    Recalling the definition~\eqref{t_m} of $t_m$,
    \[
    x_{N_m} < \lambda s_m+\phi < x_{N_m+1}.
    \]
    Once again, Theorem~\ref{thm:M and delta} yields 
    \[
    \delta\le\dist(t_m,\Z) = \min\{t_m-N_m,N_m+1-t_m\},
    \]
    which gives
    \[
    \delta\le t_m-N_m \qquad \text{and} \qquad \delta\le N_m+1-t_m \implies t_m-N_m\le 1-\delta.
    \]
    This completes the proof.
\end{proof}

\begin{lemma}\label{lem:bound}
    There exist positive constants $C_1,C_2$ such that for $r\in\{-1,0,1,2\}$ and $m\in \M$,
     \[
    C_1N_m^{-1/2} \le \abs{x_{N_m+r}-(\lambda s_m+\phi)} \le C_2N_m^{-1/2}.
    \]
\end{lemma}

\begin{proof}
    For $r\in\{-1,0,1,2\}$ and $m\in\M$, we have by~\eqref{t_m} and~\eqref{x_n}
    \[
    (\lambda s_m+\phi)-x_{N_m+r} = (\lambda\sqrt{t_m}+\phi)-(\lambda\sqrt{N_m+r}+\phi) = \lambda\frac{t_m-N_m-r}{\sqrt{t_m}+\sqrt{N_m+r}}.
    \]
    Thus, by Lemma~\ref{lem:ineq}(a) and (c), it follows that
    \begin{align*}
    \lambda\frac{\delta}{\sqrt{t_m}+\sqrt{N_m+r}} \le \abs{x_{N_m+r}-(\lambda s_m+\phi)} \le \lambda\frac{2-\delta}{\sqrt{t_m}+\sqrt{N_m+r}}.
    \end{align*}
    Moreover, from Lemma~\ref{lem:ineq}(a), 
    \[
    \begin{aligned}
    & \sqrt{N_m} < \sqrt{N_m}+\sqrt{N_m+r} < \sqrt{t_m}+\sqrt{N_m+r},\\ 
    &\begin{aligned}
    \sqrt{t_m}+\sqrt{N_m+r} < \sqrt{N_m+1}+\sqrt{N_m+r}
    \le \sqrt{N_m}\sqrt{1+\frac{1}{N_m}}+\sqrt{N_m}\sqrt{1+\frac{2}{N_m}}
    < 4\sqrt{N_m}\,.
    \end{aligned}
    \end{aligned}
    \]
    Hence,
    \[
    \frac{\lambda\delta}{4}N_m^{-1/2} \le \abs{x_{N_m+r}-(\lambda s_m+\phi)} \le \lambda(2-\delta)N_m^{-1/2}.
    \]
    Therefore, the lemma holds with
    \[
    C_1 = \frac{\lambda\delta}{4}, \qquad 
    C_2 = \lambda(2-\delta).
    \]\vspace{-10pt}
\end{proof}
    By the definition of $s_m$ in~\eqref{s_m}, we have
    \[
    \lambda s_m+\phi = \frac{\pi}{2}+m\pi.
    \]
    Thus, for any $n\in\Z_{\ge0}$, we may write 
    \begin{equation}\label{eqn:F_sin}
    F\bigl(\sqrt{n}\bigr) = \cos\left(\frac{\pi}{2}+m\pi+x_n-(\lambda s_m+\phi)\right) = (-1)^{m+1}\sin\left(x_n-(\lambda s_m+\phi)\right).
    \end{equation}
    
\begin{lemma}\label{lem:sign and bounds}
    For all sufficiently large $m\in\M$,
    \begin{align*}
     \sgn \bigl(F\bigl(\sqrt{N_m}\bigr)\bigr) = (-1)^m, 
     \qquad
     \sgn\bigl(F\bigl(\sqrt{N_m+1}\bigr)\bigr) = (-1)^{m+1},
    \end{align*}
    and there exist positive constants $c_1,c_2$ such that for $r\in\{-1,0,1,2\}$,
    \[
    c_1N_m^{-1/2} \le \abs{F\bigl(\sqrt{N_m+r}\bigr)} \le c_2N_m^{-1/2}.
    \]
\end{lemma}

\begin{proof}
     Note that if $\abs{h} < 1$, then 
     \begin{align}\label{ineq:h-sin}
     \sgn(\sin h) = \sgn(h) 
     \qquad
     \text{and}
     \qquad
     \frac{2}{\pi}\abs{h} \le \abs{\sin h} \le \abs{h}.
     \end{align}
     Let $C_1,C_2$ be as in Lemma~\ref{lem:bound}, and fix $r\in\{-1,0,1,2\}$. 
     By Lemma~\ref{lem:bound},
    \begin{align*}
    \abs{x_{N_m+r}-(\lambda s_m+\phi)} \le C_2N_m^{-1/2},
    \end{align*}
    so for sufficiently large $m$,
    \begin{equation*}
    \abs{x_{N_m+r}-(\lambda s_m+\phi)} < 1. 
    \end{equation*}
     Moreover, Lemma~\ref{lem:ineq}(b) implies 
     \begin{align*}
      & x_{N_m}-(\lambda s_m+\phi) < 0, \\
      & x_{N_m+1}-(\lambda s_m+\phi) > 0,
     \end{align*} 
     and Lemma~\ref{lem:bound} implies
      \[
     C_1N_m^{-1/2} \le \abs{x_{N_m+r}-(\lambda s_m+\phi)} \le C_2N_m^{-1/2}.
     \]
     Therefore, the desired sign and size bounds follow using~\eqref{eqn:F_sin} and~\eqref{ineq:h-sin} with
      \[
     c_1 = \frac{2}{\pi}C_1, \qquad c_2 = C_2.
     \]
\end{proof}

\begin{lemma}\label{lem:inq:abs F}
    There exists $\epsilon>0$ such that for all sufficiently large $m\in\M$,
    \begin{align*}
    \abs{F\bigl(\sqrt{N_m-1}\bigr)} > (1+\epsilon)\abs{F\bigl(\sqrt{N_m}\bigr)} \quad \text{ and }
    \quad \abs{F\bigl(\sqrt{N_m+2}\bigr)} > (1+\epsilon)\abs{F\bigl(\sqrt{N_m+1}\bigr)}.
    \end{align*}
\end{lemma}

\begin{proof}
    By Lemma~\ref{lem:ineq}(a),
    \begin{align}\label{ineq:abs N-t}
    \abs{N_m-1-t_m} \ge 2\abs{N_m-t_m} > 0 \quad \text{and} \quad
    \abs{N_m+2-t_m} \ge 2\abs{N_m+1-t_m} > 0.
    \end{align}
    Using the Taylor expansion
    \[
    \sqrt{1+x} = 1+O(\abs{x}),
    \]
    as $x\to0$, we obtain 
    \begin{equation}\label{Big O:t}
        \sqrt{t_m} = \sqrt{N_m}\sqrt{1+\frac{t_m-N_m}{N_m}} = \sqrt{N_m}\left(1+O\left(N_m^{-1}\right)\right) = \sqrt{N_m} + O\left(N_m^{-1/2}\right).
    \end{equation}
    For $r\in\{0,2\}$, we have
    \begin{equation}\label{Big O:N_N-1}
        \sqrt{N_m+r} - \sqrt{N_m+r-1} = \frac{1}{\sqrt{N_m+r}+\sqrt{N_m+r-1}}=O\left(N_m^{-1/2}\right).
    \end{equation}
    Using~\eqref{ineq:abs N-t},~\eqref{Big O:t}, and~\eqref{Big O:N_N-1}, it follows that
    \begin{subequations}
    \label{est:N_m-1--n_m o}
     \begin{align}
        & \abs{\frac{x_{N_m-1}-(\lambda s_m+\phi)}{x_{N_m}-(\lambda s_m+\phi)}} = \abs{\frac{N_m-1-t_m}{N_m-t_m}}\left(1+\frac{\sqrt{N_m}-\sqrt{N_m-1}}{\sqrt{N_m-1}+\sqrt{t_m}}\right) \ge 2+o(1), \\
        & \abs{\frac{x_{N_m+2}-(\lambda s_m+\phi)}{x_{N_m+1}-(\lambda s_m+\phi)}} = \abs{\frac{N_m+2-t_m}{N_m+1-t_m}}\left(1+\frac{\sqrt{N_m+1}-\sqrt{N_m+2}}{\sqrt{N_m+2}+\sqrt{t_m}}\right) \ge 2+o(1).
    \end{align}   
    \end{subequations}
    From~\eqref{eqn:F_sin},~\eqref{ineq:h-sin}, and Lemma~\ref{lem:bound}, 
    \begin{align*}
    \abs{x_{N_m+r}-(\lambda s_m+\phi)} \ge \abs{F\bigl(\sqrt{N_m+r}\bigr)} \ge \frac{2}{\pi}\abs{x_{N_m+r}-(\lambda s_m+\phi)},
    \end{align*}
    for $r\in\{-1,0,1,2\}$ and all sufficiently large $m\in\M$.
    Together with~\eqref{est:N_m-1--n_m o}, it follows that there exists $\epsilon>0$ satisfying
    \begin{subequations}
    \begin{align*}
    & \abs{F\bigl(\sqrt{N_m-1}\bigr)} \ge (1+\epsilon)\abs{F\bigl(\sqrt{N_m}\bigr)}, \\
    & \abs{F\bigl(\sqrt{N_m+2}\bigr)} \ge (1+\epsilon)\abs{F\bigl(\sqrt{N_m+1}\bigr)},
    \end{align*}   
    \end{subequations}
    for all sufficiently large $m\in\M$.   
\end{proof}
\begin{remark}
    All the results in this section hold for any $\lambda > 0$ such that $\pi^2/\lambda^2$ is irrational and any $\phi\in\R$, provided one restricts to sufficiently large $m$ so that $s_m \ge 0$. 
\end{remark}

\section{Exceptional Sign Pairs and Local Minima}\label{sec:final}

    In this section, we combine the results of the previous section with the asymptotic formula for $V(n)$ to prove the existence of infinitely many exceptional sign pairs and local minima.
    
\begin{theorem}\label{thm:Gen}
    Let $V(n)$ be any integer sequence that satisfies~\eqref{V_n}.
    For all sufficiently large $m\in\M$, $V(N_m)$ and $V(N_{m}+1)$ have the same sign.
    Furthermore, at least one of $\abs{V(N_m)}$ and $\abs{V(N_{m}+1)}$ is a local minimum of the sequence $(\abs{V(n)})_{n\ge0}$. 
\end{theorem}

\begin{proof}
    For $n>0$, write
    \[
    V(n) = \MM^*(n)+\EE(n),
    \]
    where $(\MM^*(n))_{n > 0}$ and $(\EE(n))_{n > 0}$ are sequences of real numbers satisfying
    \begin{align}
     & \MM^*(n) = (-1)^{n}\alpha\frac{e^{\beta\sqrt{n}}}{\sqrt{n}}F\bigl(\sqrt{n}\bigr)\left(1+O\left(n^{-1/2}\right)\right),\label{est:M*} \\
     & \EE(n) = O\left(\frac{e^{\gamma\sqrt{n}}}{\sqrt{n}}\right).\label{est:E}
    \end{align}
    Since $\left(1+O\left(n^{-1/2}\right)\right) > 0$ for sufficiently large $n$, 
    \[
    \sgn(\MM^*(n)) = \sgn(\MM(n)),
    \]
    with $(\MM(n))_{n>0}$ defined by
    \[
    \MM(n)\coloneqq(-1)^{n}\alpha\frac{e^{\beta\sqrt{n}}}{\sqrt{n}}F\bigl(\sqrt{n}\bigr).
    \]
    Thus, it follows from Lemma~\ref{lem:sign and bounds} that 
    \begin{equation}\label{eq:sgn M*}
    \sgn(\MM^*(N_m)) = \sgn(\MM^*(N_m+1)) = (-1)^{N_m+m}\sgn(\alpha),
    \end{equation}
    for sufficiently large $m\in\M$.
    Moreover, combining Lemma~\ref{lem:sign and bounds} with 
    \[
    \sqrt{N_m} \le \sqrt{N_m+1} \le 2\sqrt{N_m},
    \]
    we obtain that there exist positive constants $A_1,\,A_2$ such that
    \begin{equation}\label{inq:M}
    A_1\frac{e^{\beta\sqrt{N_m+r}}}{N_m+r} \le \abs{\MM
    (N_m+r)} \le A_2\frac{e^{\beta\sqrt{N_m+r}}}{N_m+r},
    \end{equation}
    for $r\in\{0,1\}$. 
    Indeed, one may take $A_1 = c_1\abs{\alpha} \text{ and } A_2 = 2c_2\abs{\alpha}$.
    
    Since $\gamma < \beta$, we have
    \[
    \frac{e^{\gamma\sqrt{n}}}{\sqrt{n}} = o\left(\frac{e^{\beta\sqrt{n}}}{n}\right).
    \]
    Hence, it follows from~\eqref{est:E} and~\eqref{inq:M} that for $r\in\{0,1\}$ and all sufficiently large $m\in\M$,
    \[
    \abs{\EE(N_m+r)} < \frac{1}{2}\abs{\MM(N_m+r)} \le \abs{\MM^*(N_m+r)}.
    \]
    The last inequality is because
    \[
    \left(1+O\left(n^{-1/2}\right)\right) \ge \frac{1}{2},
    \]
    for $n$ large enough. Consequently, for all such $m\in\M$,
    \[
    \sgn(V(N_m)) = \sgn(\MM^*(N_m)),\qquad \sgn(V(N_m+1)) = \sgn(\MM^*(N_m+1)).
    \]
    Together with~\eqref{eq:sgn M*}, this proves that
    \[
    \sgn(V(N_m)) = \sgn(V(N_m+1)).
    \]
    
    We now turn to the local minimum property. It suffices to show that for all sufficiently large $m\in\M$,
    \[
    \abs{V(N_m)} < \abs{V(N_m-1)} 
    \qquad
    \text{and}
    \qquad
    \abs{V(N_m+1)} < \abs{V(N_m+2)}.
    \]
    Recall from~\eqref{Big O:N_N-1} that for $r\in\{0,2\}$,
    \begin{equation*}
    \sqrt{N_m+r} - \sqrt{N_m+r-1} = \frac{1}{\sqrt{N_m+r}+\sqrt{N_m+r-1}}=O\left(N_m^{-1/2}\right).
    \end{equation*}
    Observe that using this, and the Taylor expansions
    \[
    e^x=1+O(\,\abs{x}), \qquad (1+x)^{1/2} = 1+O(\,\abs{x}), \qquad (1+x)^{-1/2} = 1+O(\,\abs{x}),
    \]
    as $x\to 0$, we obtain
    \begin{subequations}
     \begin{align}
    & e^{\beta\left(\sqrt{N_m-1}-\sqrt{N_m}\right)} = e^{O(N_m^{-1/2})} = 1+O(N_m^{-1/2}), \label{Big O:e_0}\\
    & e^{\beta\left(\sqrt{N_m+2}-\sqrt{N_m+1}\right)} = e^{O(N_m^{-1/2})} = 1+O(N_m^{-1/2}), \label{Big O:e_2}\\
    &\frac{\sqrt{N_m}}{\sqrt{N_m-1}} = \left(1-\frac{1}{N_m}\right)^{-1/2} = 1+O(N_m^{-1}), \label{Big O:Sqrt_0}\\
    &\frac{\sqrt{N_m+1}}{\sqrt{N_m+2}} = \left(1-\frac{1}{N_m+2}\right)^{1/2} = 1+O(N_m^{-1}). \label{Big O:Sqrt_2}
    \end{align}   
    \end{subequations}
    From~\eqref{Big O:e_0} and~\eqref{Big O:Sqrt_0},
    \begin{subequations}\label{est:exp_term}
    \begin{align}
    \frac{e^{\beta\sqrt{N_m-1}}}{\sqrt{N_m-1}} = \frac{e^{\beta\sqrt{N_m}}}{\sqrt{N_m}}\left(1+O\left(N_m^{-1/2}\right)\right),
    \end{align}
    \text{and from~\eqref{Big O:e_2} and~\eqref{Big O:Sqrt_2},}
    \begin{align}
    \frac{e^{\beta\sqrt{N_m+2}}}{\sqrt{N_m+2}} = \frac{e^{\beta\sqrt{N_m+1}}}{\sqrt{N_m+1}}\left(1+O\left(N_m^{-1/2}\right)\right).  
    \end{align}
    \end{subequations}
    Hence, by~\eqref{est:M*},~\eqref{est:exp_term}, and Lemma~\ref{lem:inq:abs F}, there exists $\epsilon>0$ such that for all sufficiently large $m\in\M$,
    \begin{subequations}\label{ineq:M*}
    \begin{align}
    \abs{\MM^*(N_m-1)} 
    & = \abs{\alpha}\frac{e^{\beta\sqrt{N_m}}}{\sqrt{N_m}}\abs{F\bigl(\sqrt{N_m-1}\bigr)}\left(1+O\left(N_m^{-1/2}\right)\right)\nonumber \\
    & > \abs{\alpha}\frac{e^{\beta\sqrt{N_m}}}{\sqrt{N_m}}\abs{F\bigl(\sqrt{N_m}\bigr)}(1+\epsilon)\left(1+O\left(N_m^{-1/2}\right)\right)\nonumber \\
    & > \left(1+\frac{\epsilon}{2}\right)\abs{\MM^*(N_m)},
    \end{align}
    \begin{align}
    \qquad\quad\abs{\MM^*(N_m+2)} 
    & = \abs{\alpha}\frac{e^{\beta\sqrt{N_m+1}}}{\sqrt{N_m+1}}\abs{F\bigl(\sqrt{N_m+2}\bigr)}\left(1+O\left(N_m^{-1/2}\right)\right)\nonumber \\
    & > \abs{\alpha}\frac{e^{\beta\sqrt{N_m+1}}}{\sqrt{N_m+1}}\abs{F\bigl(\sqrt{N_m+1}\bigr)}(1+\epsilon)\left(1+O\left(N_m^{-1/2}\right)\right)\nonumber \\
    & > \left(1+\frac{\epsilon}{2}\right)\abs{\MM^*(N_m+1)}.   
    \end{align}   
    \end{subequations}
    From~\eqref{est:M*},~\eqref{est:E}, and Lemma~\ref{lem:sign and bounds}, for $r\in\{-1,0,1,2\}$,
    \[
    \frac{\EE(N_m+r)}{\MM^*(N_m+r)} = O\left(\frac{\sqrt{N_m}}{e^{(\beta-\gamma)\sqrt{N_m+r}}}\right).
    \]
    Since $\beta>\gamma$, we have for $m\in\M$ large enough,
    \begin{equation}\label{ineq:E-M*}
    \abs{\EE(N_m+r)} < \frac{\epsilon}{8}\abs{\MM^*(N_m+r)}.
    \end{equation}
    We may assume $0<\epsilon<4$ so that
    \[
    \left(1-\frac{\epsilon}{8}\right)\left(1+\frac{\epsilon}{2}\right) > \left(1+\frac{\epsilon}{8}\right).
    \]
    Therefore, using~\eqref{ineq:M*} and~\eqref{ineq:E-M*}, for sufficiently large $m\in\M$,
    \begin{align*}
    \abs{V(N_m-1)} 
    & \ge \abs{\MM^*(N_m-1)} - \abs{\EE(N_m-1)} 
      > \left(1-\frac{\epsilon}{8}\right)\left(1+\frac{\epsilon}{2}\right)\abs{\MM^*(N_m)}\\
    & > \left(1+\frac{\epsilon}{8}\right)\abs{\MM^*(N_m)}
      > \abs{\MM^*(N_m)} + \abs{\EE(N_m)}
      \ge \abs{V(N_m)},
    \end{align*}
    \begin{align*}
    \qquad\qquad\quad\abs{V(N_m+2)} 
    & \ge \abs{\MM^*(N_m+2)} - \abs{\EE(N_m+2)}
      > \left(1-\frac{\epsilon}{8}\right)\left(1+\frac{\epsilon}{2}\right)\abs{\MM^*(N_m+1)}\\
    & > \left(1+\frac{\epsilon}{8}\right)\abs{\MM^*(N_m+1)} 
      > \abs{\MM^*(N_m+1)} + \abs{\EE(N_m+1)}
      \ge \abs{V(N_m+1)}.
    \end{align*}
    Hence, $(\,\abs{V(n)})_{n\ge0}$ has a local minimum at $N_m$ or $N_m+1$. This completes the proof of the theorem.
\end{proof}

\begin{samepage} 
    The following corollary completes the proof of Theorem~\ref{thm:main}.
\begin{corollary}
    The sequences $V_2(n),V_3(n)$, and $V_4(n)$ exhibit an \emph{exceptional sign pair phenomenon} in the sense of Theorem~\ref{thm:main}.
\end{corollary}
\end{samepage}

\begin{proof}
    For each $j\in\{2,3,4\}$, the asymptotic expansion of $V_j(n)$ given in Theorem~B~\cite{kalita2026} satisfies the hypotheses of Theorem~\ref{thm:Gen} with
    \[
    \alpha=\alpha_j,\qquad
    \beta=2\,\Re(\sqrt{W_j}),\qquad
    \gamma=\Re(\sqrt{W_j}),\qquad
    \lambda=2\,\Im(\sqrt{W_j}),\qquad
    \phi=\frac{\pi}{4}.
    \]
    Also recall that $N_m\to\infty$ as $m\to\infty$ in Theorem~\ref{thm:Gen}.
    Thus, the result follows immediately from Theorem~\ref{thm:Gen} after passing to a subsequence and reindexing.
\end{proof}

\begin{remark}
    The sequence $(N_m)_{m\in\M}$ grows quadratically.
    Indeed, Lemma~\ref{lem:ineq}(a) together with~\eqref{t_m} implies
    \[
    N_m = t_m + O(1) = \left(\frac{m\pi}{\lambda}+\frac{\pi-2\phi}{2\lambda}\right)^2 + O(1)  = \frac{\pi^2}{\lambda^2}m^2 + O(m).
    \] 
\end{remark}

\begin{remark}
    Theorem~\ref{thm:Gen} remains valid for any $\lambda>0$ such that $\pi^2/\lambda^2$ is irrational and any $\phi\in\R$.
\end{remark}
    
\bibliographystyle{abbrv}
\bibliography{ref.bib}
\end{document}